\documentclass[a4paper,12pt]{article}
\usepackage[utf8]{inputenc}
\usepackage[english]{babel}
\usepackage{amsmath, amsfonts, amsthm, amssymb}
\usepackage[labelfont=bf]{caption,subcaption}
\usepackage{color,soul}
\usepackage{enumerate}
\usepackage{float}
\usepackage[margin=2.5cm]{geometry}
\usepackage{graphicx}
\usepackage{hyperref}
\usepackage[nameinlink]{cleveref}
\usepackage{mathrsfs}
\usepackage{pdfpages}
\usepackage{tikz}
\usepackage{latexsym}
\usepackage[]{algorithm2e}

\newtheorem{theorem}{Theorem}[section]
\newtheorem{lemma}[theorem]{Lemma}
\newtheorem{corollary}[theorem]{Corollary}

\theoremstyle{definition}

\newtheorem{example}[theorem]{Example}
\newtheorem{remark}[theorem]{Remark}
\definecolor{red}{RGB}{255,0,0}
\definecolor{green}{RGB}{0,255,0}
\definecolor{blue}{RGB}{0,0,255}
\definecolor{yellow}{RGB}{255,255,0}
\definecolor{cyan}{RGB}{0,255,255}
\definecolor{magenta}{RGB}{255,0,255}
\definecolor{orange}{RGB}{255,128,0}

\begin{document}

\title{Monomial ideals with arbitrarily high tiny powers}

\author{Oleksandra Gasanova}


\maketitle
\begin{abstract}
Powers of (monomial) ideals is a subject that still calls attraction in various ways. Let $I\subset \mathbb K[x_1,\ldots,x_n]$   be a monomial ideal and let $G(I)$ denote the (unique) minimal monomial generating set of $I$. How small can $|G(I^i)|$ be in terms of $|G(I)|$? We expect that the inequality $|G(I^2)|>|G(I)|$ should hold and that $|G(I^i)|$, $i\ge 2$, grows further whenever $|G(I)|\ge 2$. In this paper we will disprove this expectation and show that for any $n$ and $d$ there is an $\mathfrak m$-primary monomial ideal $I\subset \mathbb K[x_1,\ldots,x_n]$ such that $|G(I)|>|G(I^i)|$ for all $i\le d$. 
\end{abstract}

\section{Introduction}
Let $I\subset \mathbb K[x_1,\ldots,x_n]$  be a monomial ideal and let $G(I)$ denote its minimal monomial generating set. It is known (see for example \cite{monid}) that the function $f(i) = |G(I^i)|$, for large $i$, is a polynomial in $i$ of degree $l(I)-1$ with a positive leading coefficient. Here $l(I)$ denotes the analytic spread of $I$, that is, the Krull dimension of the fiber ring $F(I)$ of $I$. In particular, for all $i$ large enough we have $|G(I^{i+1})|>|G(I^i)|$ unless $I$ is a principal ideal.

But what kind of pathologies can occur if $i$ is small? How small can $|G(I^i)|$ be in terms of $|G(I)|$? This question has been explored in \cite{TS} and \cite{NG}. We intuitively expect that the inequality $|G(I^2)|>|G(I)|$ should hold and that $|G(I^i)|$, $i\ge 2$, grows further whenever $|G(I)|\ge 2$. This expectation has been disproven in \cite{TS}: the authors construct a family of ideals in $\mathbb K [x,y]$ for which $|G(I)|>|G(I^2)|$. 

In Section~2 we will generalize the above result and show that for any $n\ge 2$ and $d\ge 2$ there is an $\mathfrak m$-primary ideal $I\subset \mathbb K[x_1,\ldots,x_n]$ such that $|G(I)|>|G(I^i)|$ for all $i\le d$. This section contains an explicit construction and several examples.

In Section~3 we will discuss Theorem~3.1 of \cite{TS}. This theorem says that if a monomial ideal $I=\langle u_1,\ldots,u_m\rangle\subset\mathbb K[x,y]$ satisfies certain conditions, then $|G(I^2)|=9$. We will relax the conditions of this theorem and give a more intuitive proof.
\section{Ideals with arbitrarily high tiny powers in any number of variables}
Let $n$ and $d$ be positive integers with $n,d \ge 2$. We will construct an $\mathfrak m$-primary monomial ideal $I\subset \mathbb{K}[x_1,\ldots,x_n]$ such that $|G(I)|>|G(I^2)|,|G(I)|>|G(I^3)|,\ldots, |G(I)|> |G(I^d)|$.
We will briefly describe the idea, after which we will give all the necessary proofs.

Let $\mu:=x_1^t\cdots x_n^t$, where $t$ is yet to be determined. We start with the ideal $J=\langle x_1^{4t},x_2^{4t},\ldots, x_n^{4t},\allowbreak x_1^{2t}\mu, x_2^{2t}\mu,\ldots,x_n^{2t}\mu\rangle$. Note that the number of generators of $J^i$ only depends on $i$ and $n$ and not on $t$. Let $A(n,d):=\max\limits_{i\in\{1,\ldots,d\}}|G(J^i)|$. For fixed $n$ and $d$ it is a constant which can be found using computer algebra (for simplicity, one may set $t=1$ here).

So far we only have $2n$ generators in $J$. Our goal is to find at least $A(n,d)-2n+1$ monomials $q_1,\ldots, q_s$ such that the set $G(J)\cup\{q_1,\ldots,q_s \}$ contains no monomials dividing each other and  $(J+\langle q_1,\ldots, q_s\rangle)^i=J^i$ for any $i\ge 2$. In other words, we would like to add generators to $J$ without changing any higher powers of $J$. The resulting ideal will be denoted by $I$, and $J$ will be called its skeleton. Clearly, the more monomials we can add the better it is and this is where $t$ comes into play.
\begin{lemma}
Let $J=\langle x_1^{4t},x_2^{4t},\ldots, x_n^{4t}, x_1^{2t}\mu, x_2^{2t}\mu,\ldots,x_n^{2t}\mu\rangle$ with $\mu=x_1^t\cdots x_n^t$, as above. Let $q:=\mu^2=x_1^{2t}x_2^{2t}\cdots x_n^{2t}$ and let $Q:=\langle q \rangle$. Then $JQ\subseteq J^2$ and $Q^2\subseteq J^2$.
\end{lemma}
\begin{proof}
In order to prove that $JQ\subseteq J^2$, it is enough to show that $x_1^{4t}q\in J^2$ and $x_1^{2t}\mu q\in J^2$. Indeed, $x_1^{4t}q=x_1^{4t}\mu^2=(x_1^{2t}\mu)^2$ and $x_1^{2t}\mu q=x_1^{5t}x_2^{3t}\cdots x_n^{3t}$ is divisible by $x_1^{4t} x_2^{2t}\mu=x_1^{5t}x_2^{3t}x_3^{t}\cdots x_n^{t}$. 

In order to show that $Q^2\subseteq J^2$, it is enough to show that $q^2\in J^2$, which is true since $q^2=x_1^{4t}\cdots x_n^{4t}$ is divisible by $x_1^{4t}x_2^{4t}$. 
\end{proof}
\begin{corollary}
Let $J$ and $Q$ be as above and let $Q'\subseteq Q$. Then for any $i\ge 2$ we have $(J+Q')^i=J^i$.
\end{corollary}

\begin{proof}
$(J+Q')^i\subseteq (J+Q)^i=J^i+J^{i-1}Q+\ldots+JQ^{i-1}+Q^i$. 

If $j$ is even, then $J^{i-j}Q^j=J^{i-j}(Q^2)^{\frac{j}{2}}\subseteq J^{i-j}(J^2)^{\frac{j}{2}}=J^i$.

If $j$ is odd, then $J^{i-j}Q^j=J^{i-j}(Q^2)^{\frac{j-1}{2}}Q\subseteq J^{i-j}(J^2)^{\frac{j-1}{2}}Q=J^{i-1}Q=J^{i-2}(JQ)\subseteq J^{i-2}J^2=J^i$.

Therefore, $(J+Q')^i\subseteq J^i$ for any $i\ge 2$. The other inclusion is trivial, which finishes the proof.
\end{proof}
Now we know that monomials from $Q$ are those that could potentially add more generators to $J$, but do not change the higher powers of $J$. We want to choose some monomials $q_1,\ldots,q_s\in Q$ such that the set $G(J)\cup\{q_1,\ldots,q_s \}$ contains no monomials dividing each other. In other words, we need to find monomials $q_1,\ldots, q_s\in Q$ that satisfy the following three conditions:
\begin{enumerate}
\item no monomial from $\{q_1,\ldots,q_s\}$ is divisible by any monomial in $G(J)$;
\item no monomial from $G(J)$ is divisible by any monomial in $\{q_1,\ldots, q_s\}$;
\item monomials in $\{q_1,\ldots,q_s\}$ do not divide each other.
\end{enumerate}
An obvious way to have the first condition satisfied is to consider only monomials from $Q\backslash J$. This set has a nice description.
\begin{lemma}
Let $J$ and $Q$ be as above. Then 
$$Q\backslash J=\{x_1^{\alpha_1}\cdots x_n^{\alpha_n}: (\alpha_1,\ldots,\alpha_n)\in [2t,3t-1]^n\cap \mathbb{N}^n\}$$
\end{lemma} 
\begin{proof}
$\supseteq$: All monomials within the given hypercube are divisible by $q$ and none of them is divisible by any of the minimal generators of $J$ since every minimal generator of $J$ has an exponent greater than or equal to $3t$. This is the only inclusion we will use in the future construction, but we can show the other inclusion as well.

$\subseteq$: Let $x_1^{\alpha_1}\cdots x_n^{\alpha_n}\in Q\backslash J$. Then $\alpha_i\ge 2t$ for all $i$. But $x_1^{\alpha_1}\cdots x_n^{\alpha_n}\not\in J$, that is, in particular, it is not divisible by $x_1^{2t}\mu=x_1^{3t}x_2^t\cdots x_n^t$. This implies $\alpha_1\le 3t-1$. Analogously, $\alpha_i\le 3t-1$ for all $i$.
\end{proof}
\begin{figure}[H]
\centering
\begin{tikzpicture}[scale=0.4]
\draw[step=1cm,gray,opacity=0.3,line width=0.001 mm] (0,0) grid (20,20);

\draw[step=5cm,opacity=0.3] (0,0) grid (20,20);

\draw[violet, thick] (0,20)--(5,20)--(5,15)--(15,15)--(15,5)--(20,5)--(20,0);

\draw[red, thick](10,20)--(10,10)--(20,10);
\draw[red, thick, dashed] (10,14)--(14,14)--(14,10);
\foreach \x in {10,...,14}{
\foreach \y in {10,...,14}{
\fill[red] (\x,\y) circle (0.1 cm);

}
}
    \node [below, very thin] at (0,0) {$_0 $};
    \node [left, very thin] at (0,0) {$_0 $};
    \node [below, very thin] at (5,0) {$_t $};
    \node [left, very thin] at (0,5) {$_t $};
     \node [below, very thin] at (10,0) {$_{2t} $};
    \node [left, very thin] at (0,10) {$_{2t} $};
    \node [below, very thin] at (15,0) {$_{3t} $};
    \node [left, very thin] at (0,15) {$_{3t} $};
    \node [below, very thin] at (20,0) {$_{4t} $};
    \node [left, very thin] at (0,20) {$_{4t} $};
\end{tikzpicture}

\caption{monomials in $\textcolor{red}Q\backslash \textcolor{violet}J$, $n=2$.} 
\end{figure}

Now we know that any subset of monomials from $[2t,3t-1]^n$ satisfies the first condition. It is also quite obvious that any subset of monomials from $[2t,3t-1]^n$ satisfies the second condition. The only thing to be taken care of is that the chosen monomials from $[2t,3t-1]^n$ do not divide each other. The most natural way to do so is to choose monomials of the same degree. To get as many of them as possible, we should choose monomials on a central integer cross-section of this hypercube, that is, monomials of the form 
$$\left\{x_1^{\alpha_1}\cdots x_n^{\alpha_n}: (\alpha_1,\ldots,\alpha_n)\in [2t,3t-1]^n\cap \mathbb{N}^n, \alpha_1+\ldots+\alpha_n=\left\lfloor\frac{n(5t-1)}{2}\right\rfloor\right\}.$$
Note that if $n(t-1)$ is even, we have a unique central integer cross-section, otherwise there are two of them giving the same number of integer points; the other central integer cross-section can be obtained by replacing $\lfloor.\rfloor$ by $\lceil.\rceil$ in the expression above.
The number of integer points on every central integer cross-section of $[2t,3t-1]^n$ equals the number of integer points on every central integer cross-section of $[0,t-1]^n$ and equals the central (and largest) coefficient(s) in the expansion of $(1+x+\ldots+x^{t-1})^n$.
For a fixed $n$ the number of these monomials only depends on $t$ and can be made arbitrarily large for $t$ large enough. 

To summarize all of the above, we need to perform the following steps:
\begin{enumerate}
\item Fix $n$ and $d$.
\item Let $J=\langle x_1^{4t},x_2^{4t},\ldots, x_n^{4t}, x_1^{2t}\mu, x_2^{2t}\mu,\ldots,x_n^{2t}\mu\rangle$ with $\mu=x_1^t\cdots x_n^t$, as above. Compute the number of generators in $J,J^2,J^3,\ldots, J^d$ using computer algebra. These numbers are independent of $t$, so we may set $t=1$ in this step. Take the maximum of these numbers and call it $A(n,d)$.
\item We already have $2n$ generators in $J$; we would like to have at least $A(n,d)+1$, that is, we need at least $A(n,d)-2n+1$ more. There exists $t$ such that the number of integer points on each central integer cross-section of $[2t,3t-1]^n$ is greater or equal to $A(n,d)-2n+1$. Choose any such $t$ and add all the appropriate monomials to our skeleton $J$. This is our ideal $I$.
\end{enumerate}
Let us discuss a few examples demonstrating the algorithm above.
\begin{example}
\label{ex}
Let us first consider an easy case with a small number of variables.
\begin{enumerate}
\item We will fix $n=2$ and $d=6$.
\item Let $J=\langle x^{4t},y^{4t},x^{3t}y^t,x^ty^{3t}\rangle$. In this step we may set $t=1$. Computing the powers of this ideal up to the sixth, we obtain: $|G(J^2)|=9$, $|G(J^3)|=13$, $|G(J^4)|=17$, $|G(J^5)|=21$, $|G(J^6)|=25$. Thus $A(2,6)=25$.
\item We have $4$ generators, but we would like to have at least $26$. That is, we need to add at least $22$ more. A square of the form $[2t,3t-1]^2$ has $t$ integer points on the diagonal, thus we can choose $t=22$. Therefore, our skeleton is $J=\langle x^{88},y^{88},x^{66}y^{22},x^{22}y^{66}\rangle$.  The monomials we want to add are on the diagonal of $[2t,3t-1]^2=[44,65]^2$, that is, $x^{65}y^{44}, x^{64}y^{45}, \ldots, x^{45}y^{64}, x^{44}y^{65}$.

Therefore,
\begin{multline*}
 I=\langle x^{88},x^{66}y^{22},x^{65}y^{44},x^{64}y^{45},x^{63}y^{46},x^{62}y^{47},x^{61}y^{48},x^{60}y^{49},x^{59}y^{50},\\x^{58}y^{51},x^{57}y^{52},x^{56}y^{53},x^{55}y^{54},x^{54}y^{55},x^{53}y^{56},x^{52}y^{57},x^{51}y^{58},\\x^{50}y^{59},x^{49}y^{60},x^{48}y^{61},x^{47}y^{62},x^{46}y^{63},x^{45}y^{64},x^{44}y^{65},x^{22}y^{66},y^{88}\rangle.
\end{multline*} 
\end{enumerate}
We see that $|G(I)|=26$, $|G(I^2)|=9$, $|G(I^3)|=13$, $|G(I^4)|=17$, $|G(I^5)|=21$, $|G(I^6)|=25$, as desired.
\end{example}
\begin{example}
Here is another example.
\begin{enumerate}
\item Fix $n=3$ and $d=3$. 
\item Let $J=\langle x^{4t},y^{4t},z^{4t},x^{3t}y^tz^t,x^ty^{3t}z^t,x^ty^tz^{3t}\rangle$. In this step we may set $t=1$. Computing the powers of this ideal up to the third, we obtain: $|G(J^2)|=18$, $|G(J^3)|=34$. Thus $A(3,3)=34$.
\item We have $6$ generators, but we would like to have at least $35$. That is, we need to add at least $29$ more. 
As we already know, the number of integer points on every central integer cross-section of $[2t,3t-1]^3$ equals the number of integer points on every central integer cross-section of $[0,t-1]^3$ and equals the central (and largest) coefficient(s) in the expansion of $(1+x+\ldots+x^{t-1})^3$. An explicit computation shows that $t=7$ is the smallest suitable integer and we can add $37$ monomials. Thus our skeleton is $J=\langle x^{28},y^{28},z^{28},x^{21}y^7z^7,x^7y^{21}z^7,x^7y^7z^{21}\rangle$. The monomials we want to add are those on the (unique) central integer cross-section of $[2t,3t-1]^3=[14,20]^3$. Up to a shift, it is the same as the central integer cross-section of $[0,6]^3$. The points we are looking for satisfy $\alpha_1+\alpha_2+\alpha_3=9$, $\alpha_i\le 6$. All the possible triples are: $(6,3,0)$ - 6 points considering all the permutations, $(6,2,1)$ - 6 points, $(5,4,0)$ - 6 points, $(5,3,1)$ - 6 points, $(5,2,2)$ - 3 points, $(4,4,1)$ - 3 points, $(4,3,2)$ - 6 points, $(3,3,3)$ - 1 point. This gives us $37$ extra monomials, as desired. Shifting them back and adding to our ideal (they are written in the same order as above) gives us 
\begin{multline*}
I=\langle x^{28},y^{28},z^{28},x^{21}y^7z^7,x^7y^{21}z^7, x^7y^7z^{21},\\x^{20}y^{17}z^{14},x^{20}y^{14}z^{17},x^{17}y^{20}z^{14},x^{17}y^{14}z^{20},x^{14}y^{20}z^{17},x^{14}y^{17}z^{20},\\ x^{20}y^{16}z^{15},x^{20}y^{15}z^{16},x^{16}y^{20}z^{15},x^{16}y^{15}z^{20},x^{15}y^{20}z^{16},x^{15}y^{16}z^{20},\\ x^{19}y^{18}z^{14}, x^{19}y^{14}z^{18},x^{18}y^{19}z^{14},x^{18}y^{14}z^{19},x^{14}y^{19}z^{18},x^{14}y^{18}z^{19},\\x^{19}y^{17}z^{15},x^{19}y^{15}z^{17},x^{17}y^{19}z^{15},x^{17}y^{15}z^{19},x^{15}y^{19}z^{17},x^{15}y^{17}z^{19},\\x^{19}y^{16}z^{16},x^{16}y^{19}z^{16},x^{16}y^{16}z^{19},x^{18}y^{18}z^{15},x^{18}y^{15}z^{18},x^{15}y^{18}z^{18},\\ x^{18}y^{17}z^{16},x^{18}y^{16}z^{17},x^{17}y^{18}z^{16},x^{17}y^{16}z^{18},x^{16}y^{18}z^{17},x^{16}y^{17}z^{18},\\x^{17}y^{17}z^{17}\rangle.
\end{multline*} 

\end{enumerate}
We see that $|G(I)|=43$, $|G(I^2)|=18$, $|G(I^3)|=34$, as desired.
\end{example}
\section{Improved conditions for tiny squares}
Let $I\subset \mathbb{K}[x,y]$ be a monomial ideal and let $G(I)$ denote the minimal monomial generating set of $I$. Then $G(I)=\{u_1,\ldots,u_m\}$, where $u_i=x^{a_i}y^{b_i}$ for all $i$ and where the exponents $a_i,b_i\in \mathbb{N}$ satisfy $a_1>a_2>\ldots>a_m$ and $b_1<b_2<\ldots<b_m$.

Let $V:=\{(i,j)\in \mathbb{N}^2:1\le i\le j\le m\}$ and consider the map
\begin{align*}
f: V &\rightarrow I^2\\
(i,j)&\mapsto u_iu_j.
\end{align*}
Then $f(V)$ generates $I^2$, but it is not a minimal generating set in general.
\begin{example}
Let $I=\langle x^4,x^3y^2,y^3\rangle$. Then $G(I^2)=\{x^8,x^7y^2,x^4y^3,x^3y^5,y^6\}$. Note that $u_2^2\not\in G(I^2)$ since $u_1u_3|u_2^2$. The picture below represents $V$; $(i,j)$ is marked with a star if $f(i,j)\in G(I^2)$ and $(i,j)$ is marked with a usual dot if $f(i,j)\not\in G(I^2)$ and the arrows show which monomials divide which. If several monomials are equal, we can choose one that will be marked with a star and mark the others with dots.
\begin{figure}[H]
\centering
\begin{tikzpicture}
\draw[very thick,->] (0,0) -- (3.5,0);
\draw[very thick,->] (0,0) -- (0,3.5);
 \foreach \x in {1,...,3}{
    \node [below, very thin] at (\x,0) {$_\x$};
    \node [left, very thin] at (0,\x) {$_\x$};
    \node at (1,\x) {$\ast$};
    \node at (\x,3) {$\ast$};
    }
 \fill (2,2) circle (0.07 cm);   
 \draw[very thick,->] (1.05,2.95) -- (1.95,2.05);
\end{tikzpicture}
\caption{generators of $I^2$.}
\end{figure}
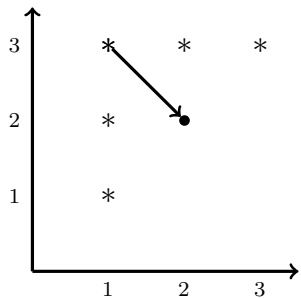
\end{example}

\begin{theorem}{(Theorem~3.1 from \cite{TS})}
Let $m\ge 5$ and let $I=\langle u_1,\ldots,u_m\rangle$ be an ideal with $u_i=x^{a_i}y^{b_i}$ for all $i$, where $a_1>\ldots>a_m$ and $b_1<\ldots<b_m$. Assume that the following divisibility conditions hold:

\begin{align}
u_1u_m &|u_2u_{m-1}               \tag{2}\label{2}\\
u_1u_{m-1}&|u_2u_3      \tag{3.1}\label{3.1}\\
u_1u_{m-1}&|u_{m-2}^2      \tag{3.2}\label{3.2}\\
u_2^2 &|u_1u_3         \tag{4.1}\label{4.1}\\
u_2^2 &|u_1u_{m-2}         \tag{4.2}\label{4.2}\\
u_2u_m &|u_3u_{m-1}\tag{5.1}\label{5.1}\\
u_2u_m &|u_{m-2}u_{m-1}\tag{5.2}\label{5.2}\\
u_{m-1}^2 &|u_3u_m    \tag{6.1}\label{6.1}\\
u_{m-1}^2 &|u_{m-2}u_m.    \tag{6.2}\label{6.2}
\end{align}
Then $G(I^2)=\{u_1^2,u_1u_2,u_2^2\}\cup\{u_1u_{m-1},u_1u_m,u_2u_m\}\cup\{u_{m-1}^2,u_{m-1}u_m,u_m^2\}.$
\end{theorem}
Let us look closer at the conditions above. 

If we multiply conditions~\eqref{2} and \eqref{4.2}, we will get 
$u_1u_m\cdot u_2^2|u_2u_{m-1}\cdot u_1u_{m-2}$, that is, $u_2u_m|u_{m-2}u_{m-1}$, which is exactly condition~\eqref{5.2}.

If we multiply conditions~\eqref{2} and \eqref{4.1}, we will get 
$u_1u_m \cdot u_2^2|u_2u_{m-1} \cdot u_1u_3$, that is, $u_2u_m|u_3u_{m-1}$, which is exactly condition~\eqref{5.1}.

If we multiply conditions~\eqref{2} and \eqref{6.1}, we will get 
$u_1u_m \cdot u_{m-1}^2|u_2u_{m-1} \cdot u_3u_m$, that is, $u_1u_{m-1}|u_2u_3$, which is exactly condition~\eqref{3.1}.

In other words, conditions~\eqref{5.1}, \eqref{5.2} and \eqref{3.1} follow from the other conditions and are redundant.
We are now left with conditions~\eqref{2}, \eqref{3.2}, \eqref{4.1}, \eqref{4.2}, \eqref{6.1} and \eqref{6.2}. This set of conditions has a nice property of being "almost self-dual" in the following sense. Recall that $V=\{(i,j)\in \mathbb{N}^2:1\le i\le j\le m\}$. This set of points is symmetric with respect to the line $i+j=m+1$. If $(i,j)\in V$, then $(m+1-j,m+1-i)\in V$ is obtained by reflecting $(i,j)$ about the line $i+j=m+1$. If $u_iu_j\in f(V)$, then $u_{m+1-j}u_{m+1-i}\in f(V)$ will be called  its \textbf{dual}. If $i+j=m+1$, the corresponding monomial will be  called \textbf{self-dual}.

Consider condition~\eqref{2}. If we dualize all monomials in this condition, we will get it back.

Consider condition~\eqref{4.1}. If we dualize all monomials in this condition, we will get condition~\eqref{6.2}.

Consider condition~\eqref{4.2}. If we dualize all monomials in this condition, we will get condition~\eqref{6.1}.

The only condition that has no dual is condition~\eqref{3.2}. Intuitively, we would like to use a self-dual set of conditions, that is, a set of conditions such that if we dualize every condition, we get the same set of conditions. Our set of conditions is not self-dual. At the first glance, it can be resolved in several ways:
\begin{enumerate}
\item It could be the case that condition~\eqref{3.2} follows from other conditions. However, this is not the case, as Remark~\ref{counter-ex} shows.
\item We can add the dual of condition~\eqref{3.2} to our set. However, this seems unnatural, given that the theorem holds without adding any other conditions.
\item We can remove condition~\eqref{3.2} and try to prove  the theorem without it. This is exactly what we will do in Theorem~\ref{newts}. We will use the following relabelling of conditions: $\eqref{2}\rightarrow \eqref{A}$, $\eqref{4.1}\rightarrow\eqref{B}$, $\eqref{6.2}\rightarrow \eqref{B*}$, $\eqref{4.2}\rightarrow\eqref{C}$, $\eqref{6.1}\rightarrow\eqref{C*}$. But before proving Theorem~\ref{newts}, we need a preliminary lemma.
\end{enumerate}
\begin{lemma}{(Lemma~2.3 from \cite{TS})}
\label{lemma}
Let $v,v_1,v_2\in V$. Assume that $v_1\le v_2$ and that $f(v)$ divides both $f(v_1),f(v_2)$. Then $f(v)$ divides $f(v')$ for all $v'\in V$ such that $v_1\le v'\le v_2$.
\end{lemma}
\begin{theorem}{(Improved conditions for tiny squares)}
\label{newts}
Let $m\ge 5$ and let $I=\langle u_1,\ldots,u_m\rangle$ be an ideal with $u_i=x^{a_i}y^{b_i}$ for all $i$, where $a_1>\ldots>a_m$ and $b_1<\ldots<b_m$. Assume that the following divisibility conditions hold:
\begin{align}
u_1u_m &|u_2u_{m-1}     \tag{A} \label{A}\\
u_2^2 &|u_1u_3          \tag{B} \label{B}\\
u_{m-1}^2 &|u_{m-2}u_m  \tag{B*}\label{B*}\\
u_2^2 &|u_1u_{m-2}      \tag{C} \label{C}\\
u_{m-1}^2 &|u_3u_m.     \tag{C*}\label{C*}
\end{align}
Then $G(I^2)=\{u_1^2,u_1u_2,u_2^2\}\cup\{u_1u_{m-1},u_1u_m,u_2u_m\}\cup\{u_{m-1}^2,u_{m-1}u_m,u_m^2\}.$
\end{theorem}
\begin{proof}
In order to see that these monomials generate $I^2$, it is enough to show that each monomial in $f(V)$ is divisible by one of these nine monomials. We distinguish several cases. \Cref{proof} shows which monomials are covered by which cases.
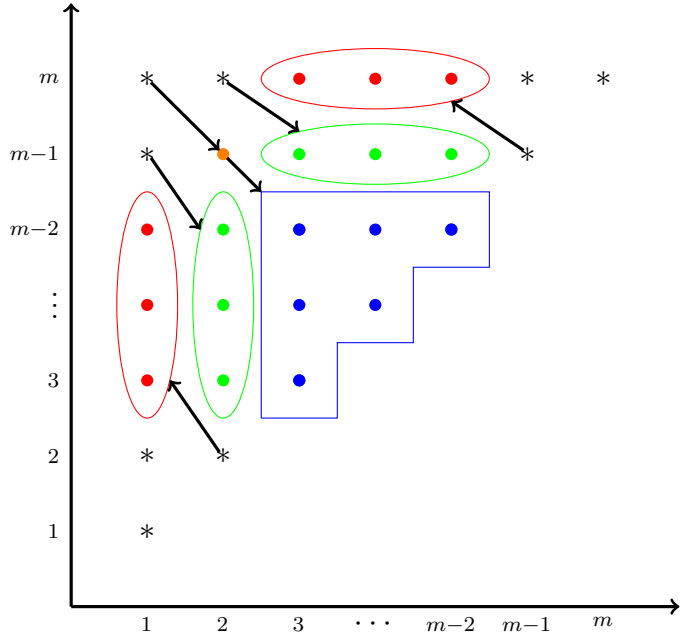
\begin{figure}[H]
\centering
\begin{tikzpicture}

\draw[very thick,->] (0,0) -- (8,0);
\draw[very thick,->] (0,0) -- (0,8);
\foreach \x in {1,...,3}{
    \node [below, very thin] at (\x,0) {$_\x$};
    \node [left, very thin] at (0,\x) {$_\x$};
  }
   
\node [below, very thin] at (7,0) {$_m$};
\node [below, very thin] at (6,0) {$_{m-1}$};
\node [below, very thin] at (5,0) {$_{m-2}$};
\node [below, very thin] at (4,0) {$\cdots$};
\node [left, very thin] at (0,4.1) {$_{\vdots}$};
\node [left, very thin] at (0,5) {$_{m-2}$};
\node [left, very thin] at (0,6) {$_{m-1}$};
\node [left, very thin] at (0,7) {$_{m}$};   
\node at (1,1) {$\ast$};
\node at (1,2) {$\ast$};
\node at (2,2) {$\ast$};
\node at (1,6) {$\ast$};
\node at (1,7) {$\ast$};
\node at (2,7) {$\ast$};
\node at (6,6) {$\ast$};
\node at (6,7) {$\ast$};
\node at (7,7) {$\ast$};

\draw [red] (1,4) ellipse (0.4 cm and 1.5 cm);
\draw [red] (4,7) ellipse (1.5 cm and 0.4 cm);
\draw [green] (2,4) ellipse (0.4 cm and 1.5 cm);
\draw [green] (4,6) ellipse (1.5 cm and 0.4 cm);
\fill[orange] (2,6) circle (0.08 cm);
\foreach \x in {3,...,5}{
\fill[red] (1,\x) circle (0.08 cm);
\fill[red] (\x,7) circle (0.08 cm);
\fill[green] (2,\x) circle (0.08 cm);
\fill[green] (\x,6) circle (0.08 cm);
\fill[blue] (\x,\x) circle (0.08 cm);
\fill[blue] (3,\x) circle (0.08 cm);
\fill[blue] (\x,5) circle (0.08 cm);
}
\draw[blue] (2.5,2.5)--(2.5,5.5)--(5.5,5.5)--(5.5,4.5)--(4.5,4.5)--(4.5,3.5)--(3.5,3.5)--(3.5,2.5)--(2.5,2.5);
\draw[very thick,->] (1.05,6.95) -- (1.95,6.05);
\draw[very thick,->] (2.05,5.95)--(2.5,5.5);
\draw[very thick,->] (1.05,5.95)--(1.7,5);
\draw[very thick,->] (2.05,6.95)--(3,6.3);
\draw[very thick,->] (1.95,2.05)--(1.3,3);
\draw[very thick,->] (5.95,6.05)--(5,6.7);
\end{tikzpicture}
\caption{illustration of the proof of \Cref{newts}.}
\label{proof}
\end{figure}

\textcolor{orange}{Case 0} (self-dual): $i=2,j=m-1$. We are done by condition~\eqref{A}.

\textcolor{red}{Case 1}: $(1,3)\le(1,j)\le(1,m-2)$. Conditions~\eqref{B} and~\eqref{C}, together with Lemma~\ref{lemma}, imply $u_2^2|u_1u_j$ for all $3\le j\le m-2$.

\textcolor{red}{Case 1*} (dual to Case 1): $(3,m)\le(i,m)\le(m-2,m)$. By the dual argument (that is, using conditions~\eqref{B*} and~\eqref{C*}) and Lemma~\ref{lemma} we conclude that $u_{m-1}^2|u_iu_m$ for all $3\le i\le m-2$.

\textcolor{green}{Case 2}: $(2,3)\le(2,j)\le(2,m-2)$.
If we multiply conditions~\eqref{A} and~\eqref{B*}, we obtain
$$u_1u_m\cdot u_{m-1}^2|u_2u_{m-1}\cdot u_{m-2}u_m\Leftrightarrow u_1u_{m-1}|u_2u_{m-2}.$$
If we multiply conditions~\eqref{A} and~\eqref{C*}, we obtain
$$u_1u_m\cdot u_{m-1}^2|u_2u_{m-1}\cdot u_3u_m\Leftrightarrow u_1u_{m-1}|u_2u_3.$$
Combining these two divisibility conditions with Lemma~\ref{lemma}, we conclude that $u_1u_{m-1}|u_2u_j$ for all $3\le j\le m-2$.

\textcolor{green}{Case 2*}(dual to case 2): $(3,m-1)\le(i,m-1)\le(m-2,m-1)$. We apply arguments dual to those from Case 2, that is, we multiply conditions~\eqref{A} and~\eqref{B} and we multiply conditions~\eqref{A} and~\eqref{C}. Combining these two divisibility conditions with Lemma~\ref{lemma}, we conclude that $u_2u_m|u_iu_{m-1}$ for all $3\le i\le m-2$.

\textcolor{blue}{Case 3} (self-dual): $(3,3)\le (i,j)\le(m-2,m-2)$. If we multiply conditions~\eqref{A}, \eqref{B} and \eqref{C*}, we obtain
$$u_1u_m\cdot u_2^2\cdot u_{m-1}^2|u_2u_{m-1}\cdot u_1u_3\cdot u_3u_m\Leftrightarrow u_2u_{m-1}|u_3^2.$$
Dually, if we multiply conditions~\eqref{A}, \eqref{B*} and \eqref{C}, we obtain $u_2u_{m-1}|u_{m-2}^2$. Combining these two divisibility conditions with Lemma~\ref{lemma}, we conclude that $u_2u_{m-1}|u_iu_j$ for all $(3,3)\le (i,j)\le(m-2,m-2)$. 

So far we know that $|G(I^2)|\le 9$. The proof of $|G(I^2)|\ge 9$ can be found in \cite{TS}. In any case, we are not so interested in proving that $|G(I^2)|\ge 9$, the essential point here is $|G(I^2)|\le 9$.
\end{proof}

\begin{example}{(a three-parameter family of ideals with tiny squares)}
\label{family}
Let $l, k, t$ be positive integers such that $k\ge 4t$. Let $I=\langle u_1,\ldots,u_{tl+4}\rangle$, where $u_i=x^{a_i}y^{b_i}$ with $(a_1, \ldots a_{tl+4})=(b_{tl+4},\ldots,b_1)=(kl, (k-t)l, (k-t)l-1,\ldots,(k-2t)l, tl, 0)$. Clearly, $kl>(k-t)l>(k-t)l-1$ and $(k-2t)l>tl>0$ under the condition that $k\ge 4t$. Also, $(k-t)l-1\ge(k-2t)l\Leftrightarrow tl\ge 1$, that is, it is indeed a decreasing sequence ($u_{tl+2}$ is $u_3$ if $t=l=1$, but there is no problem).
We check the conditions for tiny squares:

Condition~\eqref{A} holds trivially.

Condition~\eqref{B}: $u_2^2|u_1u_3 \Leftrightarrow x^{2(k-t)l}y^{2tl}|x^{kl}\cdot x^{(k-t)l-1}y^{(k-2t)l}\Leftrightarrow 2(k-t)l\le(2k-t)l-1 \text{ and } 2tl\le(k-2t)l\Leftrightarrow tl\ge 1 \text { and } k\ge 4t$.

Condition~\eqref{B*} holds by the symmetry of our ideal.

Condition~\eqref{C}: $u_2^2|u_1u_{m-2} \Leftrightarrow x^{2(k-t)l}y^{2tl}|x^{kl}\cdot x^{(k-2t)l}y^{(k-t)l-1}\Leftrightarrow 2(k-t)l\le 2(k-t)l \text{ and } 2tl\le (k-t)l-1\Leftrightarrow (k-3t)l\ge 1$.

Condition~\eqref{C*} holds by the symmetry of our ideal.
 
\end{example}
\begin{remark}
\label{counter-ex}
Consider Example~\ref{family} again. Put $l=1$ and $k=4t$, then $I=\langle u_1,\ldots,u_{t+4}\rangle$, where $u_i=x^{a_i}y^{b_i}$ with $(a_1, \ldots a_{t+4})=(b_{t+4},\ldots,b_1)=(4t, 3t, 3t-1,...,2t, t, 0)$. More explicitly, $$I=\langle x^{4t},x^{3t}y^{t},x^{3t-1}y^{2t},\ldots,x^{2t}y^{3t-1},x^{t}y^{3t},y^{4t}\rangle.$$ 
First of all note that this ideal does not satisfy condition~\eqref{3.2} which requires $u_1u_{m-1}|u_{m-2}^2$. Also note that this is exactly the ideal obtained using the construction, described in Section~2 for $n=2$: the first two and the last two monomials generate the skeleton $J$ of $I$. All other monomials are those on the central integer cross-section of the hypercube $[2t,3t-1]^2$, which is simply a diagonal in this case. If we put $t=22$, we will recover the ideal from Example~\ref{ex}.
\end{remark}

\bibliographystyle{siam}
\bibliography{tiny_squares}

\begin{thebibliography}{1}

\bibitem{TS}
{\sc S.~Eliahou, J.~Herzog, and M.~M. Saem}, {\em Monomial ideals with tiny
  squares}, Journal of Algebra, 514 (2018), pp.~99--112.

\bibitem{monid}
{\sc J.~Herzog and T.~Hibi}, {\em Monomial ideals}, in Graduate Texts in
  Mathematics, vol.~260, Springer, 2011.

\bibitem{NG}
{\sc J.~Herzog, M.~M. Saem, and N.~Zamani}, {\em The number of generators of
  the powers of an ideal}, International Journal of Algebra and Computation, 29
  (2019), pp.~827--847.

\end{thebibliography}
\end{document}